\documentclass[11pt]{amsart}
\usepackage{mathrsfs,latexsym,amsfonts,amssymb}
\usepackage{hyperref}
\newtheorem{theorem}{Theorem}[section]
\newtheorem{lemma}[theorem]{Lemma}
\newtheorem{corollary}[theorem]{Corollary}
\newtheorem{question}[theorem]{Question}
\newtheorem{example}[theorem]{Example}
\theoremstyle{definition}
\newtheorem{definition}[theorem]{Definition}
\newtheorem{proposition}[theorem]{Proposition}
\theoremstyle{remark}

\begin{document}

\title[Strongly topological gyrogroups and quotient with respect to $L$-subgyrogroups]
{Strongly topological gyrogroups and quotient with respect to $L$-subgyrogroups}

\author{Meng Bao}
\address{(Meng Bao): College of Mathematics, Sichuan University, Chengdu 610064, P. R. China}
\email{mengbao95213@163.com}

\author{Xuewei Ling}
\address{(Xuewei Ling): Institute of Mathematics, Nanjing Normal University, Nanjing, Jiangsu 210046, P.R. China}
\email{781736783@qq.com}

\author{Xiaoquan Xu*}
\address{(Xiaoquan Xu): School of mathematics and statistics,
Minnan Normal University, Zhangzhou 363000, P. R. China}
\email{xiqxu2002@163.com}

\thanks{The authors are supported by the National Natural Science Foundation of China (Nos. 11661057, 12071199) and the Natural Science Foundation of Jiangxi Province, China (No. 20192ACBL20045)\\
*corresponding author}


\begin{abstract}
In this paper, some generalized metric properties in strongly topological gyrogroups are studied. In particular, it is proved that when $G$ is a strongly topological gyrogroup with a symmetric neighborhood base $\mathscr U$ at $0$ and $H$ is a second-countable admissible subgyrogroup generated from $\mathscr U$, if the quotient space $G/H$ is an $\aleph_{0}$-space (resp., cosmic space), then $G$ is also an $\aleph_{0}$-space (resp., cosmic space); If the quotient space $G/H$ has a star-countable $cs$-network (resp., $wcs^{*}$-network, $k$-network), then $G$ also has a star-countable $cs$-network (resp., $wcs^{*}$-network, $k$-network). Moreover, it is shown that when $G$ is a strongly topological gyrogroup with a symmetric neighborhood base $\mathscr U$ at $0$ and $H$ is a locally compact metrizable admissible subgyrogroup generated from $\mathscr U$, if the quotient space $G/H$ is sequential, then $G$ is also sequential; Furthermore, if the quotient space $G/H$ is strictly (strongly) Fr\'echet-Urysohn, then $G$ is also strictly (strongly) Fr\'echet-Urysohn; Finally, if the quotient space $G/H$ is a stratifiable space (semi-stratifiable space, $\sigma$-space, $k$-semistratifiable space), then $G$ is a local stratifiable space (semi-stratifiable space, $\sigma$-space, $k$-semistratifiable space).
\end{abstract}

\maketitle
\section{Introduction}
The $c$-ball of relativistically admissible velocities with the Einstein velocity addition was researched for many years. The Einstein velocity addition $\oplus _{E}$ is given as follows: $$\mathbf{u}\oplus _{E}\mathbf{v}=\frac{1}{1+\frac{\mathbf{u}\cdot \mathbf{v}}{c^{2}}}(\mathbf{u}+\frac{1}{\gamma _{\mathbf{u}}}\mathbf{v}+\frac{1}{c^{2}}\frac{\gamma _{\mathbf{u}}}{1+\gamma _{\mathbf{u}}}(\mathbf{u}\cdot \mathbf{v})\mathbf{u}),$$ where $\mathbf{u,v}\in \mathbb{R}_{c}^{3}=\{\mathbf{v}\in \mathbb{R}^{3}:||\mathbf{v}||<c\}$ and $\gamma _{\mathbf{u}}$ is given by $$\gamma _{\mathbf{u}}=\frac{1}{\sqrt{1-\frac{\mathbf{u}\cdot \mathbf{u}}{c^{2}}}}.$$ In particular, by the research of this, Ungar in \cite{UA1988,UA} posed the concept of a gyrogroup. It is obvious that a gyrogroup has a weaker algebraic structure than a group. Then, in 2017, a gyrogroup was endowed with a topology by Atiponrat \cite{AW} such that the multiplication is jointly continuous and the inverse is also continuous. At the same time, she claimed that M\"{o}bius gyrogroups, Einstein gyrogroups, and Proper velocity gyrogroups, that were studied in \cite{FM, FM1,FM2,UA}, are all topological gyrogroups. Moreover, Cai, Lin and He in \cite{CZ} proved that every topological gyrogroup is a rectifiable space and deduced that first-countability and metrizability are equivalent in topological gyrogroups. Indeed, this kind of space has been studied for many years, see \cite{AW1,AW2020,LF,LF1,LF2,LF3,SL,ST,ST1,ST2,UA2005,UA2002,WAS2020}. After then, in 2019, Bao and Lin \cite{BL} defined the concept of strongly topological gyrogroups and claimed that M\"{o}bius gyrogroups, Einstein gyrogroups, and Proper velocity gyrogroups endowed with standard topology are all strongly topological gyrogroups but not topological groups. Furthermore, they proved that every strongly topological gyrogroup with a countable pseudocharacter is submetrizable and every locally paracompact strongly topological gyrogroup is paracompact \cite{BL1,BL2}. They also claimed that every feathered strongly topological gyrogroup is paracompact, and hence a $D$-space \cite{BL}. In the same paper, they gave an example to show that there exists a strongly topological gyrogroup which has an infinite $L$-subgyrogroup. Therefore, it is meaningful to research the quotient spaces of a strongly topological gyrogroup with respect to $L$-subgyrogroups as left cosets. In particular, we investigate what properties of topological groups still valid on strongly topological gyrogroups.

In this paper, we mainly study some generalized metric properties in strongly topological gyrogroups. In Section 3, it is proved that when $G$ is a strongly topological gyrogroup with a symmetric neighborhood base $\mathscr U$ at $0$ and $H$ is a second-countable admissible subgyrogroup generated from $\mathscr U$, if the quotient space $G/H$ is an $\aleph_{0}$-space (resp., cosmic space), then $G$ is also an $\aleph_{0}$-space (resp., cosmic space); If the quotient space $G/H$ has a star-countable $cs$-network (resp., $wcs^{*}$-network, $k$-network), then $G$ also has a star-countable $cs$-network (resp., $wcs^{*}$-network, $k$-network). In Section 4, we investigate the quotient space $G/H$ with some generalized metric properties when $G$ is a strongly topological gyrogroup with a symmetric neighborhood base $\mathscr U$ at $0$ and $H$ is a locally compact metrizable admissible subgyrogroup generated from $\mathscr U$. We show that when $G$ is a strongly topological gyrogroup with a symmetric neighborhood base $\mathscr U$ at $0$ and $H$ is a locally compact metrizable admissible subgyrogroup generated from $\mathscr U$, if the quotient space $G/H$ is sequential, then $G$ is also sequential; Furthermore, if the quotient space $G/H$ is strictly (strongly) Fr\'echet-Urysohn, then $G$ is also strictly (strongly) Fr\'echet-Urysohn; Finally, if the quotient space $G/H$ is a stratifiable space (semi-stratifiable space, $\sigma$-space, $k$-semistratifiable space), then $G$ is a local stratifiable space (semi-stratifiable space, $\sigma$-space, $k$-semistratifiable space).

\section{Preliminaries}

Throughout this paper, all topological spaces are assumed to be Hausdorff, unless otherwise is explicitly stated. Let $\mathbb{N}$ be the set of all positive integers and $\omega$ the first infinite ordinal. The readers may consult \cite{AA, E, linbook, UA} for notation and terminology not explicitly given here. Next we recall some definitions and facts.

\begin{definition}\cite{UA}
Let $(G, \oplus)$ be a groupoid. The system $(G,\oplus)$ is called a {\it gyrogroup}, if its binary operation satisfies the following conditions:

\smallskip
(G1) There exists a unique identity element $0\in G$ such that $0\oplus a=a=a\oplus0$ for all $a\in G$.

\smallskip
(G2) For each $x\in G$, there exists a unique inverse element $\ominus x\in G$ such that $\ominus x \oplus x=0=x\oplus (\ominus x)$.

\smallskip
(G3) For all $x, y\in G$, there exists $\mbox{gyr}[x, y]\in \mbox{Aut}(G, \oplus)$ with the property that $x\oplus (y\oplus z)=(x\oplus y)\oplus \mbox{gyr}[x, y](z)$ for all $z\in G$.

\smallskip
(G4) For any $x, y\in G$, $\mbox{gyr}[x\oplus y, y]=\mbox{gyr}[x, y]$.
\end{definition}

Notice that a group is a gyrogroup $(G,\oplus)$ such that $\mbox{gyr}[x,y]$ is the identity function for all $x, y\in G$.

\begin{lemma}\cite{UA}\label{a}
Let $(G, \oplus)$ be a gyrogroup. Then for any $x, y, z\in G$, we obtain the following:

\begin{enumerate}
\smallskip
\item $(\ominus x)\oplus (x\oplus y)=y$. \ \ \ (left cancellation law)

\smallskip
\item $(x\oplus (\ominus y))\oplus gyr[x, \ominus y](y)=x$. \ \ \ (right cancellation law)

\smallskip
\item $(x\oplus gyr[x, y](\ominus y))\oplus y=x$.

\smallskip
\item $gyr[x, y](z)=\ominus (x\oplus y)\oplus (x\oplus (y\oplus z))$.

\smallskip
\item $(x\oplus y)\oplus z=x\oplus (y\oplus gyr[y,x](z))$.
\end{enumerate}
\end{lemma}

The definition of a subgyrogroup is given as follows.

\begin{definition}\cite{ST}
Let $(G,\oplus)$ be a gyrogroup. A nonempty subset $H$ of $G$ is called a {\it subgyrogroup}, denoted
by $H\leq G$, if $H$ forms a gyrogroup under the operation inherited from $G$ and the restriction of $gyr[a,b]$ to $H$ is an automorphism of $H$ for all $a,b\in H$.

\smallskip
Furthermore, a subgyrogroup $H$ of $G$ is said to be an {\it $L$-subgyrogroup}, denoted
by $H\leq_{L} G$, if $gyr[a, h](H)=H$ for all $a\in G$ and $h\in H$.
\end{definition}

The subgyrogroup criterion is given in \cite{ST}, that is, a nonempty subset $H$ of a gyrogroup $G$ is a subgyrogroup if and only if $\ominus a\in H$ and $a\oplus b\in H$ for all $a,b\in H$, which explains that by the item (4) in Lemma \ref{a} it follows the subgyrogroup criterion.

\begin{definition}\cite{AW}
A triple $(G, \tau, \oplus)$ is called a {\it topological gyrogroup} if the following statements hold:

\smallskip
(1) $(G, \tau)$ is a topological space.

\smallskip
(2) $(G, \oplus)$ is a gyrogroup.

\smallskip
(3) The binary operation $\oplus: G\times G\rightarrow G$ is jointly continuous while $G\times G$ is endowed with the product topology, and the operation of taking the inverse $\ominus (\cdot): G\rightarrow G$, i.e. $x\rightarrow \ominus x$, is also continuous.
\end{definition}

Obviously, every topological group is a topological gyrogroup. However, any topological gyrogroup whose gyrations are not identically equal to the identity is not a topological group.

\begin{definition}\cite{BL}\label{d11}
Let $G$ be a topological gyrogroup. We say that $G$ is a {\it strongly topological gyrogroup} if there exists a neighborhood base $\mathscr U$ of $0$ such that, for every $U\in \mathscr U$, $\mbox{gyr}[x, y](U)=U$ for any $x, y\in G$. For convenience, we say that $G$ is a strongly topological gyrogroup with neighborhood base $\mathscr U$ at $0$.
\end{definition}

Clearly, we may assume that $U$ is symmetric for each $U\in\mathscr U$ in Definition~\ref{d11}. Moreover, in the
classical M\"{o}bius, Einstein, or Proper Velocity gyrogroups we know that gyrations are indeed special rotations, however for an arbitrary gyrogroup, gyrations belong to the automorphism group of $G$ and need not be necessarily rotations.

In \cite{BL}, the authors proved that there is a strongly topological gyrogroup which is not a topological group, see Example \ref{lz1}.

\begin{example}\cite{BL}\label{lz1}
Let $\mathbb{D}$ be the complex open unit disk $\{z\in \mathbb{C}:|z|<1\}$. We consider $\mathbb{D}$ with the standard topology. In \cite[Example 2]{AW}, define a M\"{o}bius addition $\oplus _{M}: \mathbb{D}\times \mathbb{D}\rightarrow \mathbb{D}$ to be a function such that $$a\oplus _{M}b=\frac{a+b}{1+\bar{a}b}\ \mbox{for all}\ a, b\in \mathbb{D}.$$ Then $(\mathbb{D}, \oplus _{M})$ is a gyrogroup, and it follows from \cite[Example 2]{AW} that $$gyr[a, b](c)=\frac{1+a\bar{b}}{1+\bar{a}b}c\ \mbox{for any}\ a, b, c\in \mathbb{D}.$$ For any $n\in \mathbb{N}$, let $U_{n}=\{x\in \mathbb{D}: |x|\leq \frac{1}{n}\}$. Then, $\mathscr U=\{U_{n}: n\in \mathbb{N}\}$ is a neighborhood base of $0$. Moreover, we observe that $|\frac{1+a\bar{b}}{1+\bar{a}b}|=1$. Therefore, we obtain that $gyr[x, y](U)\subset U$, for any $x, y\in \mathbb{D}$ and each $U\in \mathscr U$, then it follows that $gyr[x, y](U)=U$ by \cite[Proposition 2.6]{ST}. Hence, $(\mathbb{D}, \oplus _{M})$ is a strongly topological gyrogroup. However, $(\mathbb{D}, \oplus _{M})$ is not a group \cite[Example 2]{AW}.
\end{example}

\bigskip
{\bf Remark 1.} Even though M\"{o}bius gyrogroups, Einstein gyrogroups, and Proper velocity gyrogroups are all strongly topological gyrogroups, all of them do not possess any non-trivial $L$-subgyrogroups. However, there is a class of strongly topological gyrogroups which has a non-trivial $L$-subgyrogroup, see the following example.

\begin{example}\label{e1}\cite{BL}
There exists a strongly topological gyrogroup which has an infinite $L$-subgyrogroup.
\end{example}

\smallskip
Indeed, let $X$ be an arbitrary feathered non-metrizable topological group, and let $Y$ be an any strongly topological gyrogroup with a non-trivial $L$-subgyrogroup (such as the gyrogroup $K_{16}$ \cite[p. 41]{UA2002}). Put $G=X\times Y$ with the product topology and the operation with coordinate. Then $G$ is an infinite strongly topological gyrogroup since $X$ is infinite. Let $H$ be a non-trivial $L$-subgyrogroup of $Y$, and take an arbitrary infinite subgroup $N$ of $X$. Then $N\times H$ is an infinite $L$-subgyrogroup of $G$.

\begin{definition}\cite{E, GMT, GJA, LS1}
Let $\mathcal{P}$ be a family of subsets of a topological space $X$.

1. $\mathcal{P}$ is called a {\it network} for $X$ if whenever $x\in U$ with $U$ open in $X$, then there exists $P\in \mathcal{P}$ such that $x\in P\subset U$.

2. $\mathcal{P}$ is called a {\it k-network} for $X$ if whenever $K\subset U$ with $K$ compact and $U$ open in $X$, there exists a finite family $\mathcal{P}^{'}\subset \mathcal{P}$ such that $K\subset \bigcup \mathcal{P}^{'}\subset U$.

3. $\mathcal{P}$ is called a {\it cs-network} for $X$ if, given a sequence $\{x_{n}\}_{n}$ converging to a point $x$ in $X$ and a neighborhood $U$ of $x$ in $X$, then $\{x\}\cup \{x_{n}:n\geq n_{0}\}\subset P\subset U$ for some $n_{0}\in \mathbb{N}$ and some $P\in \mathcal{P}$.

4. $\mathcal{P}$ is called a {\it $wcs^{*}$-network} for $X$ if, given a sequence $\{x_{n}\}_{n}$ converging to a point $x$ in $X$ and a neighborhood $U$ of $x$ in $X$, there exists a subsequence $\{x_{n_{i}}\}_{i}$ of the sequence $\{x_{n}\}_{n}$ such that $\{x_{n_{i}}:i\in \mathbb{N}\}\subset P\subset U$ for some $P\in \mathcal{P}$.

\end{definition}

It is claimed in \cite{linbook} that every base is a $k$-network and a $cs$-network for a topological space, and every $k$-network or every $cs$-network is a $wcs^{*}$-network for a topological space, but the converse does not hold. Moreover, a space $X$ has a countable $cs$-network if and only if $X$ has a countable $k$-network if and only if $X$ has a countable $wcs^{*}$-network, see \cite{LS}.

\begin{definition}\cite{ME1} Let $X$ be a topological space.

1. $X$ is called {\it cosmic} if $X$ is a regular space with a countable network.

2. $X$ is called an $\aleph_{0}$-space if it is a regular space with a countable $k$-network.
\end{definition}

{\bf Remark 2.} It was shown in \cite{GG} that every separable metric space is an $\aleph_{0}$-space. Moreover, every $\aleph_{0}$-space is a cosmic space and every cosmic space is a paracompact, separable space. Then, in \cite{GJA}, it was proved that a topological space is an $\aleph_{0}$-space if and only if it is a regular space with a countable $cs$-network.

\bigskip
Now we recall the following concept of the coset space of a topological gyrogroup.

Let $(G, \tau, \oplus)$ be a topological gyrogroup and $H$ an $L$-subgyrogroup of $G$. It follows from \cite[Theorem 20]{ST} that $G/H=\{a\oplus H:a\in G\}$ is a partition of $G$. We denote by $\pi$ the mapping $a\mapsto a\oplus H$ from $G$ onto $G/H$. Clearly, for each $a\in G$, we have $\pi^{-1}\{\pi(a)\}=a\oplus H$.
Denote by $\tau (G)$ the topology of $G$. In the set $G/H$, we define a family $\tau (G/H)$ of subsets as follows: $$\tau (G/H)=\{O\subset G/H: \pi^{-1}(O)\in \tau (G)\}.$$

The following concept of an admissible subgyrogroup of a strongly topological gyrogroup was first introduced in \cite{BL1}, which plays an important role in this paper.

A subgyrogroup $H$ of a topological gyrogroup $G$ is called {\it admissible} if there exists a sequence $\{U_{n}:n\in \omega\}$ of open symmetric neighborhoods of the identity $0$ in $G$ such that $U_{n+1}\oplus (U_{n+1}\oplus U_{n+1})\subset U_{n}$ for each $n\in \omega$ and $H=\bigcap _{n\in \omega}U_{n}$. If $G$ is a strongly topological gyrogroup with a symmetric neighborhood base $\mathscr U$ at $0$ and each $U_{n}\in \mathscr U$, we say that the admissible topological subgyrogroup is generated from $\mathscr U$ \cite{BL2}.

It was shown in \cite{BL2} that if $G$ is a strongly topological gyrogroup with a symmetric neighborhood base $\mathscr U$ at $0$, then each admissible topological subgyrogroup $H$ generated from $\mathscr U$ is a closed $L$-subgyrogroup of $G$.

\section{Quotient with respect to second-countable admissible subgyrogroups}

In this section, we study the quotient space $G/H$ with some generalized metric properties when $G$ is a strongly topological gyrogroup with a symmetric neighborhood base $\mathscr U$ at $0$ and $H$ is a second-countable admissible subgyrogroup of $G$ generated from $\mathscr U$. Suppose that $G$ is a strongly topological gyrogroup with a symmetric neighborhood base $\mathscr U$ at $0$ and $H$ is a second-countable admissible subgyrogroup generated from $\mathscr U$. We prove that if the quotient space $G/H$ is an $\aleph_{0}$-space (resp., cosmic space), then $G$ is also an $\aleph_{0}$-space (resp., cosmic space). Moreover, we show that if the quotient space $G/H$ has a star-countable $cs$-network (resp., $wcs^{*}$-network, $k$-network), then $G$ also has a star-countable $cs$-network (resp., $wcs^{*}$-network, $k$-network).

\begin{lemma}\label{t00000}\cite{BL}
Let $(G, \tau, \oplus)$ be a topological gyrogroup and $H$ an $L$-subgyrogroup of $G$. Then the natural homomorphism $\pi$ from a topological gyrogroup $G$ to its quotient topology on $G/H$ is an open and continuous mapping.
\end{lemma}

\begin{lemma}\cite{AA}\label{t00001}
Suppose that $f:X\rightarrow Y$ is an open continuous mapping of a space $X$ onto a space $Y$, $x\in X$, $B\subset Y$, and $f(x)\in \overline{B}$. Then $x\in \overline{f^{-1}(B)}$. In particular, $\overline{f^{-1}(B)}=f^{-1}(\overline{B})$.
\end{lemma}

\begin{proposition}\label{t00002}
Suppose that $G$ is a topological gyrogroup and $H$ is a closed and separable $L$-subgyrogroup of $G$. If $Y$ is a separable subset of $G/H$, $\pi ^{-1}(Y)$ is also separable in $G$.
\end{proposition}

\begin{proof}
Let $\pi$ be the natural homomorphism from $G$ onto the quotient space $G/H$. Since $Y$ is a separable subset of $G/H$, there is a countable subset $B$ of $G/H$ such that $Y\subset \overline{B}$. For each $y\in B$, we can find $x\in G$ such that $\pi (x)=y$. Since $H$ is separable and $\pi^{-1}(\pi (x))=x\oplus H$ is homeomorphic to $H$, there is a countable subset $M_{y}$ of $\pi^{-1}(\pi (x))$ such that $\overline{M_{y}}=x\oplus H$. Put $M=\bigcup \{M_{y}:y\in B\}$. It is clear that $M$ is countable and $\overline{M}=\pi ^{-1}(B)$. It follows from Lemma \ref{t00000} that $\pi$ is an open and continuous mapping. Then, $\pi^{-1}(Y)\subset \pi^{-1}(\overline{B})=\overline{\pi ^{-1}(B)}=\overline{\overline{M}}=\overline{M}$ by Lemma \ref{t00001}. Therefore, $\pi ^{-1}(Y)$ is separable in $G$.
\end{proof}

\begin{lemma}\cite{BL2}\label{t00003}
Every locally paracompact strongly topological gyrogroup is paracompact.
\end{lemma}

\begin{lemma}\label{t00004}\cite{BD}
Every star-countable family $\mathcal{P}$ of subsets of a topological space $X$ can be expressed as $\mathcal{P}=\bigcup \{\mathcal{P} _{\alpha}:\alpha \in \Lambda\}$, where each subfamily $\mathcal{P}_{\alpha}$ is countable and $(\bigcup \mathcal{P} _{\alpha})\cap (\bigcup \mathcal{P} _{\beta})=\emptyset$ whenever $\alpha \not =\beta$.
\end{lemma}

\begin{theorem}\label{3dl3}
Let $G$ be a strongly topological gyrogroup with a symmetric neighborhood base $\mathscr U$ at $0$ and let $H$ be a second-countable admissible subgyrogroup generated from $\mathscr U$. If the quotient space $G/H$ is a local $\aleph_{0}$-space (resp., locally cosmic space), then $G$ is a topological sum of $\aleph_{0}$-subspace (resp., cosmic subspaces).
\end{theorem}

\begin{proof}
We just prove the case of $\aleph_{0}$-space, and the case of cosmic space is similar.

Let $\mathscr{U}$ be a symmetric neighborhood base at $0$ such that $gyr[x, y](U)=U$ for any $x, y\in G$ and $U\in\mathscr{U}$. Since $H$ is an admissible subgyrogroup generated from $\mathscr U$, there exits a sequence $\{U_{n}:n\in \omega\}$ of open symmetric neighborhoods of the identity $0$ in $G$ such that $U_{n}\in \mathscr U$, $U_{n+1}\oplus (U_{n+1}\oplus U_{n+1})\subset U_{n}$ for each $n\in \omega$ and $H=\bigcap _{n\in \omega}U_{n}$. By the hypothesis, the quotient space $G/H$ is a local $\aleph_{0}$-space. Then we can find an open neighborhood $Y$ of $H$ in $G/H$ such that $Y$ has a countable $cs$-network. Put $X=\pi ^{-1}(Y)$. Since the natural homomorphism $\pi$ from $G$ onto $G/H$ is an open and continuous mapping, $X$ is an open neighborhood of the identity element $0$ in $G$. It follows from Proposition \ref{t00002} that $X$ is separable. Therefore, there is countable subset $B=\{b_{m}:m\in \mathbb{N}\}$ of $X$ such that $\overline{B}=X$.

By the first-countability of $H$, there exists a countable family $\{V_{n}:n\in \mathbb{N}\}\subset \mathscr U$ of open symmetric neighborhoods of $0$ in $G$ such that $V_{n+1}\oplus (V_{n+1}\oplus V_{n+1})\subset V_{n}\subset X$ for each $n\in \mathbb{N}$ and the family $\{V_{n}\cap H:n\in \mathbb{N}\}$ is a local base at $0$ for $H$. Since $Y$ is an $\aleph_{0}$-space, there is a countable $cs$-network $\{\mathcal{P}_{k}:k\in \mathbb{N}\}$ for $Y$.

{\bf Claim 1.} $X$ is an $\aleph_{0}$-space.

Put $\mathcal{F}=\{\pi ^{-1}(P_{k})\cap (b_{m}\oplus V_{n}):k,m,n\in \mathbb{N}\}$. Then $\mathcal{F}$ is a countable family of subsets of $X$. Suppose that $\{x_{i}\}_{i}$ is a sequence converging to a point $x$ in $X$ and $U$ be a neighborhood of $x$ in $X$. Then $U$ is also a neighborhood of $x$ in $G$. Let $V$ be an open neighborhood of $0$ in $G$ such that $x\oplus (V\oplus V)\subset U$. Since $\{V_{n}\cap H:n\in \mathbb{N}\}$ is a local base at $0$ for $H$, there is $n\in \mathbb{N}$ such that $V_{n}\cap H\subset V\cap H$. Moreover, $(x\oplus V_{n+1})\cap X$ is a non-empty open subset of $X$ and $\overline{B}=X$, whence $B\cap (x\oplus V_{n+1})\not =\emptyset$. Therefore, there exists $b_{m}\in B$ such that $b_{m}\in x\oplus V_{n+1}$. Furthermore, $(x\oplus V_{n+1})\cap (x\oplus V)$ is an open neighborhood of $x$ and $\pi :G\rightarrow G/H$ is an open mapping, so $\pi ((x\oplus V_{n+1})\cap (x\oplus V))$ is an open neighborhood of $\pi (x)$ in the space $Y$ and the sequence $\{\pi (x_{i})\}_{i}$ converges to $\pi (x)$ in $Y$. It is obtained that $$\{\pi (x)\}\cup \{\pi (x_{i}):i\geq i_{0}\}\subset P_{k}\subset \pi ((x\oplus V_{n+1})\cap (x\oplus V))~~for ~~some~~i_{0},k\in \mathbb{N}.$$

{\bf Subclaim 1.} $(x\oplus V_{n+1})\cap (x\oplus V)=x\oplus (V_{n+1}\cap V)$.

For every $t\in (x\oplus V_{n+1})\cap (x\oplus V)$, there are $u_{1}\in V_{n+1},v_{1}\in V$ such that $t=x\oplus u_{1}=x\oplus v_{1}$. By Lemma \ref{a}, $u_{1}=\ominus x\oplus (x\oplus u_{1})=\ominus x\oplus (x\oplus v_{1})=v_{1}$. Therefore, $(x\oplus V_{n+1})\cap (x\oplus V)\subset x\oplus (V_{n+1}\cap V)$.

On the contrary, for any $s\in x\oplus (V_{n+1}\cap V)$, there is $u_{2}\in V_{n+1}\cap V$ such that $s=x\oplus u_{2}$. It is obvious that $s\in (x\oplus V_{n+1})\cap (x\oplus V)$, that is, $x\oplus (V_{n+1}\cap V)\subset (x\oplus V_{n+1})\cap (x\oplus V)$. Hence, $(x\oplus V_{n+1})\cap (x\oplus V)=x\oplus (V_{n+1}\cap V)$.

{\bf Subclaim 2.} $\pi ^{-1}(P_{k})\cap (b_{m}\oplus V_{n+1})\subset U$.

For an arbitrary $z\in \pi ^{-1}(P_{k})\cap (b_{m}\oplus V_{n+1})$, $\pi (z)\in P_{k}\subset \pi ((x\oplus V_{n+1})\cap (x\oplus V))$. Then, since $z\in ((x\oplus V_{n+1})\cap (x\oplus V))\oplus H=(x\oplus (V_{n+1}\cap V))\oplus H$, and $H$ is an admissible subgyrogroup generated from $\mathscr U$, we have
\begin{eqnarray}
z&\in&(x\oplus (V_{n+1}\cap V))\oplus H\nonumber\\
&=&x\oplus ((V_{n+1}\cap V)\oplus gyr[(V_{n+1}\cap V),x](H))\nonumber\\
&=&x\oplus ((V_{n+1}\cap V)\oplus gyr[(V_{n+1}\cap V),x](\bigcap_{m\in \mathbb{N}}U_{m}))\nonumber\
\end{eqnarray}

\begin{eqnarray}
&\subset&x\oplus ((V_{n+1}\cap V)\oplus \bigcap_{m\in \mathbb{N}}gyr[(V_{n+1}\cap V),x](U_{m}))\nonumber\\
&=&x\oplus ((V_{n+1}\cap V)\oplus \bigcap_{m\in \mathbb{N}}U_{m})\nonumber\\
&=&x\oplus ((V_{n+1}\cap V)\oplus H).\nonumber\
\end{eqnarray}

Therefore, $\ominus x\oplus z\in (V_{n+1}\cap V)\oplus H$. Moreover, since $z\in b_{m}\oplus V_{n+1}$ and $b_{m}\in x\oplus V_{n+1}$, it follows that
\begin{eqnarray}
z&\in&(x\oplus V_{n+1})\oplus V_{n+1}\nonumber\\
&=&x\oplus (V_{n+1}\oplus gyr[V_{n+1},x](V_{n+1}))\nonumber\\
&=&x\oplus (V_{n+1}\oplus V_{n+1}).\nonumber
\end{eqnarray}
So, $(\ominus x)\oplus z\in V_{n+1}\oplus V_{n+1}$. Hence, $(\ominus x)\oplus z\in ((V_{n+1}\cap V)\oplus H)\cap (V_{n+1}\oplus V_{n+1})$. There exist $a\in (V_{n+1}\cap V),~h\in H$ and $u_{3},v_{3}\in V_{n+1}$ such that $(\ominus x)\oplus z=a\oplus h=u_{3}\oplus v_{3}$, whence $h=(\ominus a)\oplus (u_{3}\oplus v_{3})\in V_{n+1}\oplus (V_{n+1}\oplus V_{n+1})\subset V_{n}$. Therefore, $(\ominus x)\oplus z\in (V_{n+1}\cap V)\oplus (V_{n}\cap H)$, and consequently, $z\in x\oplus ((V_{n+1}\cap V)\oplus (V_{n}\cap H))\subset x\oplus (V\oplus V)\subset U$.

Since $b_{m}\in x\oplus V_{n+1}$, there is $u\in V_{n+1}$ such that $b_{m}=x\oplus u$, whence
\begin{eqnarray}
x&=&(x\oplus u)\oplus gyr[x,u](\ominus u)\nonumber\\
&=&b_{m}\oplus gyr[x,u](\ominus u)\nonumber\\
&\in&b_{m}\oplus gyr[x,u](V_{n+1})\nonumber\\
&=&b_{m}\oplus V_{n+1}.\nonumber
\end{eqnarray}
Therefore, there exists $i_{1}\geq i_{0}$ such that $x_{i}\in b_{m}\oplus V_{n+1}$ when $i\geq i_{1}$, whence $\{x\}\cup \{x_{i}:i\geq i_{1}\}\subset \pi ^{-1}(P_{k})\cap (b_{m}\oplus V_{n+1})$. Thus $\mathcal{F}$ is a countable $cs$-network for $X$, and we complete the proof of Claim 1.

Since strongly topological gyrogroup $G$ is homogeneous, $G$ is a local $\aleph_{0}$-space by Claim 1. Therefore, $G$ is a locally paracompact space. Furthermore, since every locally paracompact strongly topological gyrogroup is paracompact by Lemma \ref{t00003}, $G$ is paracompact. Let $\mathcal{A}$ is an open cover of $G$ by $\aleph_{0}$-subspace. Because the property of being an $\aleph_{0}$-space is hereditary, we can assume that $\mathcal{A}$ is locally finite in $G$ by the paracompactness of $G$. Moreover, as every point-countable family of open subsets in a separable space is countable, the family $\mathcal{A}$ is star-countable. Then $\mathcal{A}=\bigcup \{\mathcal{B}_{\alpha}:\alpha \in \Lambda\}$ by Lemma \ref{t00004}, where each subfamily $\mathcal{B}_{\alpha}$ is countable and $(\bigcup \mathcal{B}_{\alpha})\cap (\bigcup \mathcal{B}_{\beta})=\emptyset$ whenever $\alpha \not =\beta$. Set $X_{\alpha}=\bigcup \mathcal{B}_{\alpha}$ for each $\alpha \in \Lambda$. Then $G=\bigoplus_{\alpha \in \Lambda}X_{\alpha}$.

{\bf Claim 2.} $X_{\alpha}$ is an $\aleph_{0}$-subspace for each $\alpha \in \Lambda$.

Put $\mathcal{B}_{\alpha}=\{B_{\alpha ,n}:n\in \mathbb{N}\}$, where each $B_{\alpha ,n}$ is an open $\aleph_{0}$-subspace of $G$, and put $\mathcal{P}_{\alpha}=\bigcup _{n\in \mathbb{N}}\mathcal{P}_{\alpha ,n}$, where $\mathcal{P}_{\alpha ,n}$ is a countable $cs$-network for the $\aleph_{0}$-space $B_{\alpha ,n}$ for each $n\in \mathbb{N}$. Then $\mathcal{P}_{\alpha}$ is a countable $cs$-network for $X_{\alpha}$. Thus, $X_{\alpha}$ is an $\aleph_{0}$-space.

In conclusion, the strongly topological gyrogroup $G$ is a topological sum of $\aleph_{0}$-subspaces.
\end{proof}

\begin{corollary}
Let $G$ be a strongly topological gyrogroup with a symmetric neighborhood base $\mathscr U$ at $0$ and let $H$ be a second-countable admissible subgyrogroup generated from $\mathscr U$. If the quotient space $G/H$ is an $\aleph_{0}$-space (resp., cosmic space), $G$ is also an $\aleph_{0}$-space (resp., cosmic space).
\end{corollary}

\begin{theorem}\label{3dl2}
Let $G$ be a strongly topological gyrogroup with a symmetric neighborhood base $\mathscr U$ at $0$ and let $H$ be a second-countable admissible subgyrogroup generated from $\mathscr U$. If the quotient space $G/H$ has a star-countable $cs$-network, $G$ also has a star-countable $cs$-network.
\end{theorem}

\begin{proof}
Let $\mathscr{U}$ be a symmetric neighborhood base at $0$ such that $gyr[x, y](U)=U$ for any $x, y\in G$ and $U\in\mathscr{U}$. Since $H$ is an admissible subgyrogroup generated from $\mathscr U$, there exits a sequence $\{U_{n}:n\in \omega\}$ of open symmetric neighborhoods of the identity $0$ in $G$ such that $U_{n}\in \mathscr U$, $U_{n+1}\oplus (U_{n+1}\oplus U_{n+1})\subset U_{n}$ for each $n\in \omega$ and $H=\bigcap _{n\in \omega}U_{n}$. Since the $L$-subgyrogroup $H$ of $G$ is first-countable at the identity element $0$ of $G$, there exists a countable family $\{V_{n}:n\in \mathbb{N}\}\subset \mathscr U$ such that $(V_{n+1}\oplus (V_{n+1}\oplus V_{n+1}))\subset V_{n}$ for each $n\in \mathbb{N}$ and the family $\{V_{n}\cap H:n\in \mathbb{N}\}$ is a local base at $0$ for $H$.

Let $\mathcal{P}=\{P_{\alpha}:\alpha \in \Lambda\}$ be a star-countable $cs$-network for the space $G/H$. For each $\alpha \in \Lambda$, the family $\{P_{\alpha}\cap P_{\beta}:\beta \in \Lambda\}$ is a countable $wcs^{*}$-network for $P_{\alpha}$. Therefore, $P_{\alpha}$ is a cosmic space, and $P_{\alpha}$ is separable. Then it follows from Proposition \ref{t00002} that $\pi^{-1}(P_{\alpha})$ is separable. We can find a countable subset $B_{\alpha}=\{b_{\alpha ,m}:m\in \mathbb{N}\}$ of $\pi^{-1}(P_{\alpha})$ such that $\overline{B_{\alpha}}=\pi^{-1}(P_{\alpha})$.

Put $$\mathcal{F}=\{\pi ^{-1}(P_{\alpha})\cap (b_{\alpha ,m}\oplus V_{n}):\alpha \in \Lambda ,~~and~~m,n\in \mathbb{N}\}.$$ Then $\mathcal{F}$ is a star-countable family of $G$.

{\bf Claim.} $\mathcal{F}$ is a $cs$-network for $G$.

Let $\{x_{i}\}_{i}$ be a sequence converging to a point $x$ in $G$ and let $U$ be a neighborhood of $x$ in $G$. Choose an open neighborhood $V$ of $0$ in $G$ such that $(x\oplus (V\oplus V))\subset U$. Since $\{V_{n}\cap H:n\in \mathbb{N}\}$ is a local base at $0$ for $H$, there exists $n\in \mathbb{N}$ such that $V_{n}\cap H\subset V\cap H$. Since $\pi :G\rightarrow G/H$ is an open and continuous mapping, there are $i_{0}\in \mathbb{N}$ and $\alpha \in \Lambda$ such that $\{\pi (x)\}\cup \{\pi (x_{i}):i\geq i_{0}\}\subset P_{\alpha}\subset \pi ((x\oplus V_{n+1})\cap (x\oplus V))$. Since $x\in \pi^{-1}(P_{\alpha})$, $(x\oplus V_{n+1})\cap \pi^{-1}(P_{\alpha})$ is non-empty and open in the subspace $\pi^{-1}(P_{\alpha})$. Moreover, since $\overline{B_{\alpha}}=\pi^{-1}(P_{\alpha})$, there exists $m\in \mathbb{N}$ such that $b_{\alpha ,m}\in x\oplus V_{n+1}$.

{\bf Subclaim.} $\pi^{-1}(P_{\alpha})\cap (b_{\alpha ,m}\oplus V_{n+1})\subset U$.

For an arbitrary $z\in \pi^{-1}(P_{\alpha})\cap (b_{\alpha ,m}\oplus V_{n+1})$, $\pi (z)\in P_{\alpha}\subset \pi ((x\oplus V_{n+1})\cap (x\oplus V))$. By the proof of Theorem \ref{3dl3}, $z\in x\oplus ((V_{n+1}\cap V)\oplus H)$. Since $z\in b_{\alpha ,m}\oplus V_{n+1}$ and $b_{\alpha ,m}\in x\oplus V_{n+1}$, we have
\begin{eqnarray}
z&\in&(x\oplus V_{n+1})\oplus V_{n+1}\nonumber\\
&=&x\oplus (V_{n+1}\oplus gyr[V_{n+1},x](V_{n+1}))\nonumber\\
&=&x\oplus (V_{n+1}\oplus (V_{n+1})).\nonumber
\end{eqnarray}
Then, $(\ominus x)\oplus z\in V_{n+1}\oplus V_{n+1}$. Hence, $(\ominus x)\oplus z\in ((V_{n+1}\cap V)\oplus H)\cap (V_{n+1}\oplus V_{n+1})$. Therefore, there exist $a\in (V_{n+1}\cap V),~h\in H$ and $u_{1},u_{2}\in V_{n+1}$ such that $(\ominus x)\oplus z=a\oplus h=u_{1}\oplus u_{2}$, whence $h=(\ominus a)\oplus (u_{1}\oplus u_{2})\in V_{n+1}\oplus (V_{n+1}\oplus V_{n+1})\subset V_{n}$. It follows that $(\ominus x)\oplus z\in (V_{n+1}\cap V)\oplus (V_{n}\cap H)$. Thus $z\in x\oplus ((V_{n+1}\cap V)\oplus (V_{n}\cap H))\subset x\oplus (V\oplus V)\subset U$.

Since $b_{\alpha ,m}\in x\oplus V_{n+1}$, there is $u_{3}\in V_{n+1}$ such that $b_{\alpha ,m}=x\oplus u_{3}$. Thus,
\begin{eqnarray}
x&=&(x\oplus u_{3})\oplus gyr[x,u_{3}](\ominus u_{3})\nonumber\\
&=&b_{\alpha ,m}\oplus gyr[x,u_{3}](\ominus u_{3})\nonumber\\
&\in&b_{\alpha ,m}\oplus gyr[x,u_{3}](V_{n+1})\nonumber\\
&=&b_{\alpha ,m}\oplus V_{n+1}.\nonumber
\end{eqnarray}
Therefore, there exists $i_{1}\geq i_{0}$ such that $x_{i}\in b_{\alpha ,m}\oplus V_{n+1}$ whenever $i\geq i_{1}$, whence $\{x\}\cup \{x_{i}:i\geq i_{1}\}\subset \pi ^{-1}(P_{\alpha})\cap (b_{\alpha ,m}\oplus V_{n+1})$.

Therefore, we conclude that $G$ has a star-countable $cs$-network.
\end{proof}

\begin{theorem}\label{3dl4}
Let $G$ be a strongly topological gyrogroup with a symmetric neighborhood base $\mathscr U$ at $0$ and let $H$ be a second-countable admissible subgyrogroup generated from $\mathscr U$. If the quotient space $G/H$ has a star-countable $wcs^{*}$-network, $G$ has also a star-countable $wcs^{*}$-network.
\end{theorem}

\begin{proof}
Let $\mathscr{U}$ be a symmetric neighborhood base at $0$ such that $gyr[x, y](U)=U$ for any $x, y\in G$ and $U\in\mathscr{U}$. Since the $L$-subgyrogroup $H$ of $G$ is first-countable at the identity element $0$ of $G$, there exists a countable family $\{V_{n}:n\in \mathbb{N}\}$ of open symmetric neighborhoods of $0$ in $G$ such that $(V_{n+1}\oplus (V_{n+1}\oplus V_{n+1}))\subset V_{n}$ for each $n\in \mathbb{N}$ and the family $\{V_{n}\cap H:n\in \mathbb{N}\}$ is a local base at $0$ for $H$.

We construct $\mathcal{P}$ and $\mathcal{F}$ as the same with in Theorem \ref{3dl2}, and we show that $\mathcal{F}$ is a $wcs^{*}$-network for $G$.

Let $\{x_{i}\}_{i}$ be a sequence converging to a point $x$ in $G$ and $U$ be a neighborhood of $x$ in $G$. Choose an open neighborhood $V$ of $0$ in $G$ such that $(x\oplus (V\oplus V))\subset U$. Since $\{V_{n}\cap H:n\in \mathbb{N}\}$ is a local base at $0$ for $H$, there exists $n\in \mathbb{N}$ such that $V_{n}\cap H\subset V\cap H$. Since $\mathcal{P}$ is a $wcs^{*}$-network for $G/H$, there exists a subsequence $\{\pi (x_{i_{j}})\}_{j}$ of the sequence $\{\pi (x_{i})\}_{i}$ such that $\{\pi (x_{i_{j}}):j\in \mathbb{N}\}\subset P_{\alpha}\subset \pi ((x\oplus V_{n+1})\cap (x\oplus V))$ for some $\alpha \in \Lambda$. As the sequence $\{x_{i}\}_{i}$ converges to $x$, we have some $x_{i_{j}}\in x\oplus V_{n+2}$ for each $j\in \mathbb{N}$. Furthermore, since $x_{i_{1}}\in \pi ^{-1}(P_{\alpha})$, $(x_{i_{1}}\oplus V_{n+2})\cap \pi^{-1}(P_{\alpha})$ is non-empty and open in $\pi^{-1}(P_{\alpha})$. Then it follows from $\overline{B_{\alpha}}=\pi^{-1}(P_{\alpha})$ that there exists $m\in \mathbb{N}$ such that $b_{\alpha ,m}\in x_{i_{1}}\oplus V_{n+2}$. Then
\begin{eqnarray}
b_{\alpha ,m}&\in&x_{i_{1}}\oplus V_{n+2}\nonumber\\
&\subset&(x\oplus V_{n+2})\oplus V_{n+2}\nonumber\\
&=&x\oplus (V_{n+2}\oplus gyr[V_{n+2},x](V_{n+2}))\nonumber\\
&=&x\oplus (V_{n+2}\oplus V_{n+2}).\nonumber
\end{eqnarray}

Moreover, it is proved in Theorem \ref{3dl2} that $\pi^{-1}(P_{\alpha})\cap (b_{\alpha ,m}\oplus V_{n+1})\subset U$.

In conclusion, $G$ has a star-countable $wcs^{*}$-network.
\end{proof}

\begin{lemma}\cite{LF1}\label{t00005}
The following are equivalent for a rectifiable space.

(i) Every compact (countably compact) subset is first-countable.

\smallskip
(ii) Every compact (countably compact) subset is metrizable.
\end{lemma}

\begin{theorem}\label{3dl5}
Let $H$ be an $L$-subgyrogroup of a topological gyrogroup $G$, and suppose that all compact subspaces of $H$ and $G/H$ are metrizable. Then all compact subspaces of $G$ are metrizable as well.
\end{theorem}

\begin{proof}
Let $\pi$ be the natural homomorphism from $G$ onto its quotient space $G/H$ of left cosets. For an arbitrary $y\in G/H$, there exists a point $x\in G$ such that $\pi (x)=y$. Then $\pi ^{-1}(y)=x\oplus H$ which is homeomorphic to $H$.

Fix a compact subset $X$ of $G$, let $f$ be the restriction of $\pi$ to $X$. The compact subspace $Y=f(X)$ of the space $G/H$ is metrizable. Indeed, all compact subsets of the fibers of $f$ are metrizable. Since $X$ is compact and $f:X\rightarrow Y$ is continuous, it is clear that $f$ is closed mapping. By \cite[Lemma 3.3.23]{AA}, all compact subsets of $G$ are first-countable. Finally, it follows from Lemma \ref{t00005} that $X$ is metrizable.
\end{proof}

\begin{lemma}\cite[Lemma 2.1.6]{linbook1}\label{t00006}
Let $\mathcal{P}$ be a point-countable family of subsets of a space $X$. Then $\mathcal{P}$ is a $k$-network for $X$ if and only if it is a $wcs^{*}$-network for $X$ and each compact subset of $X$ is first-countable (or sequential).
\end{lemma}

\begin{theorem}
Let $G$ be a strongly topological gyrogroup with a symmetric neighborhood base $\mathscr U$ at $0$ and let $H$ be a second-countable admissible subgyrogroup generated from $\mathscr U$. If the quotient space $G/H$ has a star-countable $k$-network, $G$ has also a star-countable $k$-network.
\end{theorem}

\begin{proof}
Since $G/H$ has a star-countable $k$-network, it follows from Theorem \ref{3dl4} that $G$ has a star-countable $wcs^{*}$-network. By Lemma \ref{t00006}, each compact subset of $G/H$ is first-countable. Then every compact subset of $G$ is first-countable by \cite[Lemma 3.3.23]{AA} and Theorem \ref{3dl5}. Therefore, $G$ has a star-countable $k$-network by Lemma \ref{t00006}.
\end{proof}

\section{Quotient with respect to locally compact admissible $L$-subgyrogroups}

In this section, we research the quotient space $G/H$ with some generalized metric properties when $G$ is a strongly topological gyrogroup with a symmetric neighborhood base $\mathscr U$ at $0$ and $H$ is a locally compact metrizable admissible subgyrogroup generated from $\mathscr U$. Suppose that $G$ is a strongly topological gyrogroup with a symmetric neighborhood base $\mathscr U$ at $0$ and $H$ is a locally compact metrizable admissible subgyrogroup generated from $\mathscr U$. We show that if the quotient space $G/H$ is sequential, then $G$ is also sequential; If the quotient space $G/H$ is strictly (strongly) Fr\'echet-Urysohn, then $G$ is also strictly (strongly) Fr\'echet-Urysohn; Finally, if the quotient space $G/H$ is a stratifiable space (semi-stratifiable space, $\sigma$-space, $k$-semistratifiable space), then $G$ is a local stratifiable space (semi-stratifiable space, $\sigma$-space, $k$-semistratifiable space).

First, recall some concepts about convergence and the relations among them.

\begin{definition}\cite{FS}
Let $X$ be a topological space. A subset $A$ of $X$ is called {\it sequentially closed} if no sequence of points of $A$ converges to a point not in $A$. A subset $A$ of $X$ is called {\it sequentially open} if $X\setminus A$ is sequentially closed. $X$ is called {\it sequential} if each sequentially closed subset of $X$ is closed.
\end{definition}

\begin{definition}\cite{FS}
Let $X$ be a topological space. A space is called {\it Fr\'echet-Urysohn at a point $x\in X$} if for every $A\subset X$ with $x\in \overline{A}\subset X$ there is a sequence $\{x_{n}\}_{n}$ in $A$ such that $\{x_{n}\}_{n}$ converges to $x$ in $X$. A space is called {\it Fr\'echet-Urysohn} if it is Fr\'echet-Urysohn at every point $x\in X$.
\end{definition}

\begin{definition}\cite{GJ}(\cite{SF})
Let $X$ be a topological space. A space is called {\it strictly (strongly) Fr\'echet-Urysohn at a point $x\in X$} if whenever $\{A_{n}\}_{n}$ is a sequence (decreasing sequence) of subsets in $X$ and $x\in \bigcap _{n\in \mathbb{N}}\overline{A_{n}}$, there exists $x_{n}\in A_{n}$ for each $n\in \mathbb{N}$ such that the sequence $\{x_{n}\}_{n}$ converges to $x$. A space $X$ is called {\it strictly (strongly) Fr\'echet-Urysohn} if it is strictly (strongly) Fr\'echet-Urysohn at every point $x\in X$.
\end{definition}

It is well-known \cite{ME} that

(1) every first-countable space is a strictly Fr\'echet-Urysohn space;

\smallskip
(2) every strictly Fr\'echet-Urysohn space is a strongly Fr\'echet-Urysohn space;

\smallskip
(3) every strongly Fr\'echet-Urysohn space is a Fr\'echet-Urysohn space;

\smallskip
(4) every Fr\'echet-Urysohn space is a sequential space.

\begin{lemma}\cite{BL2}\label{4yl1}
Suppose that $(G, \tau, \oplus)$ is a strongly topological gyrogroup with a symmetric neighborhood base $\mathscr U$ at $0$, and suppose that $H$ is a locally compact admissible subgyrogroup generated from $\mathscr U$. Then there exists an open neighborhood $U$ of the identity element $0$ such that $\pi (\overline{U})$ is closed in $G/H$ and the restriction of $\pi $ to $\overline{U}$ is a perfect mapping from $\overline{U}$ onto the subspace $\pi (\overline{U})$, where $\pi: G\rightarrow G/H$ is the natural quotient mapping from $G$ onto the quotient space $G/H$.
\end{lemma}

\begin{theorem}\label{4dl5}
Suppose that $G$ is a strongly topological gyrogroup with a symmetric neighborhood base $\mathscr U$ at $0$. Suppose further that $H$ is a locally compact metrizable admissible subgyrogroup generated from $\mathscr U$ such that the quotient space $G/H$ is sequential, then $G$ is also sequential.
\end{theorem}

\begin{proof}
By the hypothesis, we assume that $G$ is a strongly topological gyrogroup with a symmetric neighborhood base $\mathscr U$ at $0$. It follows from Lemma \ref{4yl1} that there is an open neighborhood $U$ of the identity element $0$ in $G$ such that $\pi |_{\overline{U}}:\overline{U}\rightarrow \pi (\overline{U})$ is a perfect mapping and $\pi (\overline{U})$ is closed in $G/H$.

{\bf Claim 1.} Assume that $\{x_{n}\}_{n}$ is a sequence in $\overline{U}$ such that $\{\pi (x_{n})\}_{n}$ is a convergent sequence in $\pi (\overline{U})$. If $x$ is an accumulation point of the sequence $\{x_{n}\}_{n}$, then there is a subsequence of $\{x_{n}\}_{n}$ which converges to $x$.

Since $\pi |_{\overline{U}}$ is perfect, every subsequence of $\{x_{n}\}_{n}$ has an accumulation point in $\overline{U}$. Put $F=\pi^{-1}(\pi (x))\cap \overline{U}$. By the assumption, $\pi^{-1}(\pi (x))=x\oplus H$ is metrizable. Since every topological gyrogroup is regular, there exists a sequence $\{U_{k}\}_{k}$ of open subsets in $G$ such that $\overline{U_{k+1}}\subset U_{k}$ for each $k\in \mathbb{N}$ and $\{x\}=F\cap \bigcap_{k\in \mathbb{N}}U_{k}$. Choose a subsequence $\{x_{n_{k}}\}_{k}$ of $\{x_{n}\}_{n}$ such that $x_{n_{k}}\in U_{k}$ for each $k\in \mathbb{N}$. For an arbitrary accumulation point $p$ of a subsequence of the sequence $\{x_{n_{k}}\}_{k}$, we have $\pi (p)=\pi (x)$ and $p\in \bigcap _{k\in \mathbb{N}}\overline{U_{k}}$. Thus $p=x$. Therefore, $x$ is the unique accumulation point of every subsequence of $\{x_{n_{k}}\}_{k}$, proving that $x_{n_{k}}\rightarrow x$.

Choose an open neighborhood $V$ of $0$ such that $\overline{V}\subset U$.

{\bf Claim 2.} If $C$ is sequentially closed in $\overline{V}$, then $\pi (C)$ is closed in $\pi (\overline{V})$.

Suppose that $\{y_{n}\}_{n}$ is a sequence in $\pi (C)$ such that $y_{n}\rightarrow y$ in $\pi (\overline{V})$. Choose $x_{n}\in C$ with $\pi (x_{n})=y_{n}$ for each $n\in \mathbb{N}$. Since every subsequence of the sequence $\{x_{n}\}_{n}$ has an accumulation point, it follows from Claim 1 that there exist a point $x\in \pi^{-1}(y)$ and a subsequence $\{x_{n_{k}}\}_{k}$ of $\{x_{n}\}_{n}$ such that $x_{n_{k}}\rightarrow x$. Since $C$ is sequentially closed, we obtain $x\in C$ and $y\in \pi (C)$. Therefore, $\pi (C)$ is sequentially closed in $\pi (\overline{V})$. Since $\pi |_{\overline{U}}:\overline{U}\rightarrow \pi (\overline{U})$ is a closed mapping and $\pi (\overline{U})$ is closed in $G/H$, $\pi (\overline{V})$ is closed in $G/H$. Since $G/H$ is sequential, $\pi (\overline{V})$ is also sequential and then $\pi (C)$ is closed in $\pi (\overline{V})$.

{\bf Claim 3.} $\overline{V}$ is a sequential subspace.

Suppose on the contrary, there is a non-closed and sequentially closed subset $A$ of $\overline{V}$. Then there exists a point $x$ such that $x\in cl_{\overline{V}}(A)\setminus A$. It is clear that $cl_{\overline{V}}(A)=\overline{A}$. Let $f=\pi |_{\overline{V}}:\overline{V}\rightarrow \pi (\overline{V})$ and $B=A\cap f^{-1}(f(x))$. Since $B$ is a closed subset of $A$, $B$ is sequentially closed. Moreover, the fiber $f^{-1}(f(x))=(\pi^{-1}(\pi(x)))\cap \overline{V}$ is sequential, so $B$ is closed in $\overline{V}$. Since $x\not \in B$, there exists an open neighborhood $W$ of $x$ in $\overline{V}$ such that $\overline{W}\cap B=\emptyset$. Let $C=\overline{W}\cap A$, then $C$ is also sequentially closed as a closed subset of $A$ and $x\in \overline{C}\setminus C$. Therefore, $C\cap f^{-1}(f(x))=\overline{W}\cap B=\emptyset$, then $f(x)\in \overline{f(C)}\setminus f(C)$. So $f(C)=\pi (C)$ is not closed in $\pi (\overline{V})$ which is contradict with Claim 2.

Since $G$ is homogeneous and by Claim 3, we obtain that $G$ is a locally sequential space. Hence, $G$ is sequential space.
\end{proof}

\begin{lemma}\cite[Proposition 4.7.18]{AA}\label{i}
Suppose that $X$ is a regular space, and that $f: X\rightarrow Y$ is a closed mapping. Suppose also that $b\in X$ is a $G_{\delta}$-point in the space $F=f^{-1}(f(b))$ (i.e., the singleton $\{b\}$ is a $G_{\delta}$-set in the space $F$) and $F$ is Fr\'echet-Urysohn at $b$. If the space $Y$ is strongly Fr\'echet-Urysohn, then $X$ is Fr\'echet-Urysohn at $b$.
\end{lemma}

\begin{theorem}
Suppose that $G$ is a strongly topological gyrogroup with a symmetric neighborhood base $\mathscr U$ at $0$. Suppose further that $H$ is a locally compact metrizable admissible subgyrogroup generated from $\mathscr U$ such that the quotient space $G/H$ is strongly Fr\'echet-Urysohn. Then the space $G$ is also strongly Fr\'echet-Urysohn.
\end{theorem}

\begin{proof}
Suppose that $G$ is a strongly topological gyrogroup with a symmetric neighborhood base $\mathscr U$ at $0$. It follows from Lemma \ref{4yl1} that there is an open neighborhood $U$ of the identity element $0$ in $G$ such that $\pi |_{\overline{U}}:\overline{U}\rightarrow \pi (\overline{U})$ is a perfect mapping and $\pi (\overline{U})$ is closed in $G/H$.

Put $f=\pi |_{\overline{U}}:\overline{U}\rightarrow \pi (\overline{U})$. Then $f(\overline{U})=\pi (\overline{U})$ is strongly Fr\'echet-Urysohn. For each $b\in \overline{U}$, $f^{-1}(f(b))=\pi^{-1}(\pi (b))\cap \overline{U}=(b\oplus H)\cap \overline{U}$ is metrizable. Therefore, the singleton $\{b\}$ is a $G_{\delta}$-set in the space $f^{-1}(f(b))$. Moreover, since the quotient space $G/H$ is strongly Fr\'echet-Urysohn, the space $G$ is locally Fr\'echet-Urysohn by Lemma \ref{i}. Hence, $G$ is Fr\'echet-Urysohn. Furthermore, every Fr\'echet-Urysohn topological gyrogroup is strongly Fr\'echet-Urysohn by \cite[Corollary 5.2]{LF1}. So $G$ is strongly Fr\'echet-Urysohn.
\end{proof}

\begin{lemma}\cite{LS}\label{i1}
Suppose that $X$ is a regular space, and that $f: X\rightarrow Y$ is a closed mapping. Suppose also that $b\in X$ is a $G_{\delta}$-point in the space $F=f^{-1}(f(b))$ (i.e., the singleton $\{b\}$ is a $G_{\delta}$-set in the space $F$) and $F$ is countably compact and strictly Fr\'echet-Urysohn at $b$. If the space $Y$ is strictly Fr\'echet-Urysohn at $f(b)$, then $X$ is strictly Fr\'echet-Urysohn at $b$.
\end{lemma}

\begin{theorem}
Suppose that $G$ is a strongly topological gyrogroup with a symmetric neighborhood base $\mathscr U$ at $0$. Suppose further that $H$ is a locally compact metrizable admissible subgyrogroup generated from $\mathscr U$ such that the quotient space $G/H$ is strictly Fr\'echet-Urysohn, then $G$ is also strictly Fr\'echet-Urysohn.
\end{theorem}

\begin{proof}
By the hypothesis, we assume that $G$ is a strongly topological gyrogroup with a symmetric neighborhood base $\mathscr U$ at $0$. It follows from Lemma \ref{4yl1} that there is an open neighborhood $U$ of the identity element $0$ in $G$ such that $\pi |_{\overline{U}}:\overline{U}\rightarrow \pi (\overline{U})$ is a perfect mapping and $\pi (\overline{U})$ is closed in $G/H$.

Put $f=\pi |_{\overline{U}}:\overline{U}\rightarrow \pi (\overline{U})$. Then $f(\overline{U})=\pi (\overline{U})$ is strictly Fr\'echet-Urysohn. For each $b\in \overline{U}$, $f^{-1}(f(b))=\pi^{-1}(\pi (b))\cap \overline{U}=(b\oplus H)\cap \overline{U}$ is compact and metrizable. It follows from Lemma \ref{i1} that $\overline{U}$ is strictly Fr\'echet-Urysohn. Therefore, $G$ is locally strictly Fr\'echet-Urysohn and $G$ is strictly Fr\'echet-Urysohn.
\end{proof}

\begin{definition}\cite{CGD, LDJ}
A topological space $(X,\tau)$ is {\it semi-stratifiable} if there is a function $S:N\times \tau \rightarrow \{closed~~ subsets ~~of~~X\}$ such that

\smallskip
(a) if $U\in \tau$, then $U=\bigcup_{n=1}^{\infty}S(n,U)$;

\smallskip
(b) if $U,V\in \tau$ and $U\subset V$, then $S(n,U)\subset S(n,V)$ for each $n\in \mathbb{N}$.

The function $S$ is called a {\it semistratification of $X$}. (If, in addition, the function $S$ satisfies $U=\bigcup_{n=1}^{\infty}[S(n,U)]^{\circ}$ for each $U\in \tau$, then $S$ is called a {\it stratification of $X$} and $X$ is said to be {\it stratifiable} \cite{BCR}.)
\end{definition}

The concept of {\it $k$-semistratifiable space} was introduced in \cite{LDL}.

\begin{theorem}\label{4dl6}
Let $G$ be a strongly topological gyrogroup with a symmetric neighborhood base $\mathscr U$ at $0$. Suppose that $H$ is a locally compact metrizable admissible subgyrogroup generated from $\mathscr U$ such that the quotient space $G/H$ has property $\mathcal{P}$, where $\mathcal{P}$ is a topological property. Then the space $G$ is locally in $\mathcal{P}$ if $\mathcal{P}$ satisfies the following:

\smallskip
(1) $\mathcal{P}$ is closed hereditary;

\smallskip
(2) $\mathcal{P}$ contains point $G_{\delta}$-property, and

\smallskip
(3) let $f:X\rightarrow Y$ be a perfect mapping, if $X$ has $G_{\delta}$-diagonal and $Y$ is $\mathcal{P}$, then $X$ is $\mathcal{P}$.

\end{theorem}

\begin{proof}
Suppose that $\pi :G\rightarrow G/H$ is the canonical homomorphism. Since $G/H$ is in $\mathcal{P}$ and $\mathcal{P}$ contains point $G_{\delta}$-property, $\{H\}$ is a $G_{\delta}$-subset in $G/H$, that is, there exists a sequence $\{V_{n}:n\in \mathbb{N}\}$ of open sets in $G/H$ such that $\{H\}=\bigcap_{n\in \mathbb{N}}V_{n}$. Therefore, $H=\bigcap _{n\in \mathbb{N}}\pi^{-1}(V_{n})$. Since $H$ is a metrizable $L$-subgyrogroup of $G$, there is a family $\{W_{n}:n\in \mathbb{N}\}$ of open neighborhoods of the identity element $0$ such that $\{W_{n}\cap H:n\in \mathbb{N}\}$ is an open countable neighborhood base in $H$. Hence, $$\{0\}=\bigcap_{n\in \mathbb{N}}(W_{n}\cap H)=\bigcap_{n\in \mathbb{N}}(W_{n}\cap \pi^{-1}(V_{n})).$$ Then $G$ has point $G_{\delta}$-property. It follows from \cite{BL1} that every strongly topological gyrogroup with countable pseudocharacter is submetrizable. So $G$ has $G_{\delta}$-diagonal.

By Lemma \ref{4yl1}, there is an open neighborhood $U$ of the identity element $0$ in $G$ such that $\pi |_{\overline{U}}:\overline{U}\rightarrow \pi (\overline{U})$ is a perfect mapping and $\pi (\overline{U})$ is closed in $G/H$. Then by (1) and (3), the subspace $\overline{U}$ is in $\mathcal{P}$. Therefore, $G$ is locally in $\mathcal{P}$.
\end{proof}

Note that every stratifiable space, semi-stratifiable space and $\sigma$-space satisfies the conditions in Theorem \ref{4dl6}, respectively.

\begin{corollary}\label{4tl1}
Suppose that $G$ is a strongly topological gyrogroup with a symmetric neighborhood base $\mathscr U$ at $0$. Suppose further that $H$ is a locally compact metrizable admissible subgyrogroup generated from $\mathscr U$ such that the quotient space $G/H$ is a stratifiable space (semi-stratifiable space, $\sigma$-space). Then $G$ is a local stratifiable space (semi-stratifiable space, $\sigma$-space).
\end{corollary}

\begin{definition}\cite{LS2}
Suppose that $\{\mathscr V_{n}\}$ is a sequence of open covers of a space.

\smallskip
(1) $\{\mathscr V_{n}\}$ is said to be a $G_{\delta}$-\emph{diagonal sequence} for $X$ if $\{x\}=\bigcap_{n\in \mathbb{N}}st(x,\mathscr V_{n})$ for each $x\in X$.

\smallskip
(2) $\{\mathscr V_{n}\}$ is said to be a $KG$-\emph{sequence} for $X$ if $x_{n}\in st(a_{n},\mathscr V_{n})$ for each $n\in \mathbb{N}$, and $x_{n}\rightarrow p$, $a_{n}\rightarrow q$, then $p=q$.
\end{definition}

It was claimed in \cite{LS2} that if $f:X\rightarrow Y$ is a perfect map and $Y$ is a $k$-semistratifiable space, then $X$ is a $k$-semistratifiable space if and only if $X$ has a $KG$-sequence.

\begin{theorem}
Let $G$ be a strongly topological gyrogroup. If $G$ has point $G_{\delta}$-property, $G$ has a $KG$-sequence.
\end{theorem}

\begin{proof}
Suppose that $G$ is a strongly topological gyrogroup with a symmetric neighborhood base $\mathscr U$ at $0$. Since $G$ has point $G_{\delta}$-property, there exists a sequence $\{V_{n}\}_{n}$ of open neighborhoods of the identity element $0$ such that $\bigcap_{n\in \mathbb{N}}V_{n}=\{0\}$. For each $V_{n}$, there exists $U_{n}\in \mathscr U$ such that $U_{n}\oplus U_{n}\subset V_{n}$. Put $\mathcal{U}_{n}=\{x\oplus U_{n}:x\in G\}$. It is clear that each $\mathcal{U}_{n}$ is an open cover of $G$.

{\bf Claim.} $\{\mathcal U_{n}\}_{n\in \mathbb{N}}$ is a $KG$-sequence in $G$.

Let $p_{n}\in st(q_{n},\mathcal U_{n})$, where $\{q_{n}\}\rightarrow q$ and $p_{n}\rightarrow p$. For each $n\in \mathbb{N}$, we can find $x_{n}\in G$ such that $p_{n},q_{n}\in (x_{n}\oplus U_{n})$. Then, there are $v_{n},u_{n}\in U_{n}$ such that $p_{n}=x_{n}\oplus u_{n}$ and $q_{n}=x_{n}\oplus v_{n}$. Therefore, by Lemma \ref{a},
\begin{eqnarray}
x_{n}&=&(x_{n}\oplus v_{n})\oplus gyr[x_{n},v_{n}](\ominus v_{n})\nonumber\\
&=&q_{n}\oplus gyr[x_{n},v_{n}](\ominus v_{n})\nonumber\\
&\in &q_{n}\oplus gyr[x_{n},v_{n}](U_{n})\nonumber\\
&=&q_{n}\oplus U_{n}.\nonumber\
\end{eqnarray}
Then,
\begin{eqnarray}
p_{n}&=&x_{n}\oplus u_{n}\nonumber\\
&\in &(q_{n}\oplus U_{n})\oplus u_{n}\nonumber\\
&=&q_{n}\oplus (U_{n}\oplus gyr[U_{n},q_{n}](u_{n}))\nonumber\\
&\subset &q_{n}\oplus (U_{n}\oplus gyr[U_{n},q_{n}](U_{n}))\nonumber\\
&=&q_{n}\oplus (U_{n}\oplus U_{n})\nonumber\\
&\subset &q_{n}\oplus V_{n}.\nonumber
\end{eqnarray}
Therefore, $\ominus q_{n}\oplus p_{n}\in V_{n}$ for each $n\in \mathbb{N}$. Hence, $\ominus q_{n}\oplus p_{n}\in \bigcap_{n\in \mathbb{N}}V_{n}=\{0\}$, that is, $p=q$.

We conclude that $G$ has a $KG$-sequence.
\end{proof}

Naturally, we have the following result.

\begin{corollary}
Suppose that $G$ is a strongly topological gyrogroup with a symmetric neighborhood base $\mathscr U$ at $0$. Suppose further that $H$ is a locally compact metrizable admissible subgyrogroup generated from $\mathscr U$ such that the quotient space $G/H$ is $k$-semistratifiable, then the space $G$ is locally $k$-semistratifiable.
\end{corollary}

Finally, we pose the following questions.

\begin{question}
Let $\mathcal{P}$ be any calss of topological spaces which is closed hereditary and closed under locally finite unions of closed sets. Is every strongly topological gyrogroup which is locally in $\mathcal{P}$ in $\mathcal{P}$ ?
\end{question}

Clearly, if the question is affirmative, the result $G$ is a local stratifiable space (semi-stratifiable space, $\sigma$-space) in Corollary \ref{4tl1} will be strengthened directly.

\begin{question}
Let $G$ be a strongly topological gyrogroup with a symmetric neighborhood base $\mathscr U$ at $0$ and let $H$ be an admissible subgyrogroup generated from $\mathscr U$. Is the quotient space $G/H$ completely regular?
\end{question}

{\bf Acknowledgements}. The first author would like to express his congratulations to his supervisor Professor Xiaoquan Xu on the occasion of his 60th birthday. The authors are thankful to the anonymous referees for valuable remarks and corrections and all other sort of help related to the content of this article.

\end{document}